\documentclass[10pt]{article}

\usepackage{amsfonts}
\usepackage{amsmath}
\usepackage{stackengine}
\usepackage{caption}
\usepackage{amssymb}
\usepackage{xfrac}
\usepackage[dvipsnames]{xcolor}
\usepackage{amsthm}
\usepackage{fullpage}
\usepackage{dsfont}
\usepackage{enumerate}
\usepackage{mathtools}
\usepackage{float}
\usepackage{tikz-cd}
\usepackage{bbm}
\usepackage[linktocpage]{hyperref}
\hypersetup{
    colorlinks=true,
    citecolor=red,
    linkcolor=blue,
    urlcolor=red,
}
\usepackage[capitalize, nameinlink]{cleveref}
\usepackage{tikz}
\usetikzlibrary{shapes.geometric}
\usepackage{comment}
\usepackage{mathrsfs}
\usepackage{fourier}
\usepackage{graphicx}
\usepackage[all]{xy}
\usepackage[titletoc]{appendix}
\usepackage{titlesec}
 \usepackage{etoolbox, chngcntr}
\AtBeginEnvironment{appendices}{%
 \titleformat{\section}{\bfseries\Large}{\appendixname~\thesection:}{0.5em}{}%
 \titleformat{\subsection}{\bfseries\large}{\thesubsection}{0.5em}{}%
\counterwithin{equation}{section}
}
\usepackage[nottoc]{tocbibind}
\usepackage{cancel}
\usepackage[final]{showlabels}

\usepackage[backend=bibtex, giveninits=true, doi=false,isbn=false,url=false,style=alphabetic, maxnames=8, maxalphanames=9]{biblatex}
\usepackage[all]{xy}

\bibliography{Ref.bib}
\DeclareMathAlphabet{\mathcal}{OMS}{cmsy}{m}{n}

\graphicspath{ {fig/} }

\definecolor{myyellow}{RGB}{200, 200, 0}

\newtheorem{theorem}{Theorem}
\numberwithin{theorem}{section}
\newtheorem{prop}[theorem]{Proposition}
\newtheorem{lemma}[theorem]{Lemma}
\newtheorem{corollary}[theorem]{Corollary}

\theoremstyle{definition}
\newtheorem{definition}[theorem]{Definition}
\newtheorem*{remark}{Remark}
\newtheorem*{warn}{Warning}
\newtheorem{example}{Example}
\newtheorem{assumption}{Assumption}

\newtheorem{conjecture}{Conjecture}

\newcommand{\ext}{\mathrm{Ext}}
\newcommand{\bs}{\mathrm{BS}}
\newcommand{\kb}{\mathbbm{k}}
\renewcommand{\hom}{\mathrm{Hom}}
\newcommand{\im}{\mathrm{im } \,}

\newcommand{\gf}{\mathfrak{g}}
\newcommand{\uhom}{\underline{\mathrm{Hom}}}
\newcommand{\Ac}{\mathcal{A}}
\newcommand{\Cc}{\mathcal{C}}
\newcommand{\Oc}{\mathcal{O}}

\newcommand{\Mc}{\mathcal{M}}
\newcommand{\Uc}{\mathcal{U}}
\newcommand{\hf}{\mathfrak{h}}
\newcommand{\id}{\mathrm{id}}

\newcommand{\Db}{\mathbb{D}}
\newcommand{\Cb}{\mathbb{C}}
\newcommand{\jw}{\mathrm{JW}}

\newcommand{\sbim}{\mathbb{S}\textnormal{Bim}}

\newcommand\scalemath[2]{\scalebox{#1}{\mbox{\ensuremath{\displaystyle #2}}}}
\newcommand{\sbrac}[1]{\left[#1 \right]}
\newcommand{\abrac}[1]{\left\langle#1\right\rangle}

\newcommand{\paren}[1]{\left( #1 \right)}
\newcommand{\set}[1]{\left \{ #1 \right \}}

\newcommand{\un}[1]{\underline{#1}}

\newcommand{\hh}{\mathrm{HH}}
\newcommand{\ft}{\mathrm{FT}}
\newcommand{\hhh}{\mathrm{HHH}}
\newcommand{\h}{\mathrm{H}}
\newcommand{\cone}{\mathrm{Cone}}
\newcommand{\Zb}{\mathbb{Z}}

\newcommand\dboxed[1]{
\raisebox{-1ex}{
\begin{tikzpicture}
         \node[draw,dashed] {\text{$#1$}}; 
   \end{tikzpicture}}
}
\newcommand{\xequals}[1]{\overset{#1}{=\joinrel=\joinrel=}}
\newcommand{\xcong}[1]{\overset{#1}{\cong}}
\newcommand{\xsim}[1]{\overset{#1}{\simeq}}

\makeatletter
\DeclareRobustCommand\widecheck[1]{{\mathpalette\@widecheck{#1}}}
\def\@widecheck#1#2{%
    \setbox\z@\hbox{\m@th$#1#2$}%
    \setbox\tw@\hbox{\m@th$#1%
       \widehat{%
          \vrule\@width\z@\@height\ht\z@
          \vrule\@height\z@\@width\wd\z@}$}%
    \dp\tw@-\ht\z@
    \@tempdima\ht\z@ \advance\@tempdima2\ht\tw@ \divide\@tempdima\thr@@
    \setbox\tw@\hbox{%
       \raise\@tempdima\hbox{\scalebox{1}[-1]{\lower\@tempdima\box
\tw@}}}%
    {\ooalign{\box\tw@ \cr \box\z@}}}
\makeatother

\tikzset{
uni/.style={circle,fill,draw,inner sep=0mm,minimum size=1mm},
  midarrow/.style={postaction={decorate,decoration={markings,mark=at position #1 with {\arrow{>}}}}},
  midarrow/.default=0.5,
  midarrowrev/.style={postaction={decorate,decoration={markings,mark=at position #1 with {\arrow{<}}}}},
  midarrowrev/.default=0.5,
  dot/.style={circle,fill,draw,inner sep=0mm,minimum size=1.3mm},  
  rdot/.style={circle,fill, color=red, draw,inner sep=0mm,minimum size=1.3mm},  
  bdot/.style={circle,fill, color=blue, draw,inner sep=0mm,minimum size=1.3mm},  
  pdot/.style={circle,color=Violet,fill=Violet,draw,inner sep=0mm,minimum size=1.3mm},  
  hdot/.style={circle,fill=white,draw,inner sep=0mm,minimum size=1.3mm}, 
  rhdot/.style={circle,color=red,fill=white,draw,inner sep=0mm,minimum size=1.3mm, }, 
  bhdot/.style={circle,color=blue,fill=white,draw,inner sep=0mm,minimum size=1.3mm, }, 
  yhdot/.style={circle,color=myyellow,fill=white,draw,inner sep=0mm,minimum size=1.3mm, }, 
  rkhdot/.style={circle,color=red,fill=white,draw,inner sep=0mm,minimum size=4mm, }, 
  rbkhdot/.style={circle,color=red,fill=white,draw,inner sep=0mm,minimum size=6mm, }, 
  rbbkhdot/.style={circle,color=red,fill=white,draw,inner sep=0mm,minimum size=6.7mm, }, 
  btridot/.style={rectangle,color=blue,fill=white,draw,inner sep=0.3mm,minimum size=1.7mm}, 
    rtridot/.style={rectangle,color=red,fill=white,draw,inner sep=0.3mm,minimum size=1.7mm}, 
  every picture/.style=thick
}
\newsavebox\lowerdot
\savebox\lowerdot{%
\begin{tikzpicture}[scale=0.3,thick,baseline]
 \draw (0,-0.5) to (0,0.5);
 \node at (0,-0.5) {$\bullet$};
\end{tikzpicture}%
}
\newsavebox\upperdot
\savebox\upperdot{%
\begin{tikzpicture}[scale=0.3,thick,baseline]
 \draw (0,-0.5) to (0,0.5);
 \node at (0,0.5) {$\bullet$};
\end{tikzpicture}%
}

\newcommand{\rzero}{\ \raisebox{-2ex}{\begin{tikzpicture}[xscale=0.25,yscale=0.3,thick,baseline]
 \draw[red] (0,0) -- (0,2.5); 
\end{tikzpicture}}}

\newcommand{\bzero}{\ \raisebox{-2ex}{\begin{tikzpicture}[xscale=0.25,yscale=0.3,thick,baseline]
 \draw[blue] (0,0) -- (0,2.5); 
\end{tikzpicture}}}

\newcommand{\ds}{\ \raisebox{-2ex}{\begin{tikzpicture}[xscale=0.25,yscale=0.3,thick,baseline]
 \draw[red] (0,0) -- (0,2);
 \node[rtridot] at (0,1) {$\scalemath{0.9}{d_s}$};
\end{tikzpicture}}}

\newcommand{\dt}{\ \raisebox{-2ex}{\begin{tikzpicture}[xscale=0.25,yscale=0.3,thick,baseline]
 \draw[blue] (0,0) -- (0,2); 
 \node[btridot] at (0,1) {$\scalemath{0.9}{d_t}$};
\end{tikzpicture}}}

\newcommand{\rhunit}{\ \raisebox{0ex}{\begin{tikzpicture}[xscale=0.25,yscale=0.3,thick,baseline]
 \draw[red] (0,0) -- (0,1);
 \node[rhdot] at (0,0) {};
\end{tikzpicture}}}

\newcommand{\runit}{\ \raisebox{0ex}{\begin{tikzpicture}[xscale=0.25,yscale=0.3,thick,baseline]
 \draw[red] (0,0.5) -- (0,1.5);
 \node[dot, red] at (0,0.5) {};
\end{tikzpicture}}}
\newcommand{\runitd}{\ \raisebox{-1ex}{\begin{tikzpicture}[xscale=0.25,yscale=0.3,thick,baseline]
 \draw[red] (0,0.5) -- (0,1.5);
 \node[dot, red] at (0,0.5) {};
\end{tikzpicture}}}

\newcommand{\bunit}{\ \raisebox{0ex}{\begin{tikzpicture}[xscale=0.25,yscale=0.3,thick,baseline]
 \draw[blue] (0,0.5) -- (0,1.5);
 \node[dot, blue] at (0,0.5) {};
\end{tikzpicture}}}

\newcommand{\bcounit}{\ \raisebox{0ex}{\begin{tikzpicture}[xscale=0.25,yscale=0.3,thick,baseline]
 \draw[blue] (0,-1) -- (0,0);
 \node[dot, blue] at (0,0) {};
\end{tikzpicture}}}

\newcommand{\rcup}{\ \raisebox{0ex}{\begin{tikzpicture}[xscale=0.25,yscale=0.3,thick,baseline]
 \draw[red] (-0.7,1)  .. controls (-0.9, 0.3) and (-0.4, -0.4) .. (0, -0.4);
 \draw[red] (0.7,1)  .. controls (0.9, 0.3) and (0.4, -0.4) .. (0, -0.4);
\end{tikzpicture}}}

\newcommand{\rhcup}{\ \raisebox{0ex}{\begin{tikzpicture}[xscale=0.25,yscale=0.3,thick,baseline]
 \draw[red] (-0.7,1)  .. controls (-0.9, 0.3) and (-0.4, -0.4) .. (0, -0.4);
 \draw[red] (0.7,1)  .. controls (0.9, 0.3) and (0.4, -0.4) .. (0, -0.4);
 \node[rhdot] at (0,-0.4) {};
\end{tikzpicture}}}

\newcommand{\rect}{
\begin{tikzpicture}[scale=0.3]
   \draw (0,0) rectangle (2.4,0.6);
\end{tikzpicture} 
}

\title{\textbf{Serre Duality and the Whitehead Link}}
\date{\relax }
\author{Cailan Li}

\begin{document}

\maketitle
\begin{abstract}
   We prove the action of the full twist  is a Serre functor in the homotopy category of dihedral Soergel Bimodules. Our proof relies on the representability of a certain partial trace functor, which we prove using the diagrammatics for Hochschild cohomology of Soergel Bimodules. Leveraging these representability results, we then compute the triply-graded link homology of the Whitehead link. 
\end{abstract}

\tableofcontents


\section{Introduction}

Categorification has become an essential perspective in modern representation theory and low-dimensional topology. By replacing algebraic structures with richer categorical analogues, one obtains invariants and structures that reflect deeper geometric or homological phenomena. 

\subsection{History}
The study of braid group actions on triangulated categories has a long and well-documented history, bridging together representation theory, topology, and geometry, see \cite{ST01}, \cite{KS02}, \cite{RZ03}. Seeking to categorify such actions, Rouquier \cite{Rou04} constructs a categorification of the braid group $B_W$ associated to a Coxeter group $W$, which associates to $\beta\in W$ a (Rouquier) complex $F^\bullet_\beta\in K^b(\sbim(\hf, W))$ in the homotopy category of Soergel bimodules. For $W=S_n$, Khovanov \cite{Kho07} applies the functor of Hochschild cohomology to $F^\bullet_\beta$, obtaining a complex whose cohomology yields a triply-graded link homology $\overline{\hhh}$ categorifying the HOMFLY–PT polynomial. \\

In \cite{Kal09}, Kálmán established a relation between the “top” and “bottom” parts of the HOMFLY–PT polynomial, expressed in terms of the full twist braid $\ft_n$ on $n$ strands. This relationship was subsequently categorified by Gorsky, Hogancamp, Mellit, Nakagane in \cite{GHMN19}, where the authors proved an analogue of Serre duality for $ K^b(\sbim(\Cb^n, S_n))$. A central step in their argument involved a form of relative Serre duality, which they further conjectured should hold for arbitrary Coxeter groups $(W, S)$ and realizations $\hf$. The setting is as follows.

\begin{itemize}
    \item $R:=\mathrm{Sym}^\bullet (V^*(-2))$ and $R^e=R\otimes_\kb R$. $R^e-\mathrm{gmod}$ are the same as graded $R$ bimodules. 
    \item $\sbim(\hf, W)$ denotes the full subcategory of $R^e-\mathrm{gmod}$ defined in \cite[Section 3.4]{SC} consisting of the Soergel Bimodules for $W$ associated to the realization $\hf$. It is monoidally generated by $B_s, \ s\in S$. 
    \item Given a subset $I\subset S$, the full, graded, monoidal, idempotent complete subcategory of $\sbim(\hf, W)$ generated by $B_s, s\in I$ is isomorphic to $\sbim(\hf, W_I)$.
    \item $\ft_S:=F^\bullet_{w_0}\otimes_R F^\bullet_{w_0} $ where $F^\bullet_{w_0}$ is the Rouquier complex associated to the longest element in $W$. $\ft_{S/I}:=\ft_{S}\otimes_R\ft_{I}^{-1}\cong \ft_{I}^{-1}\otimes_R\ft_{S}$.
\end{itemize}

\begin{conjecture}[Relative Serre Duality]
\label{ghmnconj}
   Given a realization $\hf$ of $W$, let $\pi_L, \pi_R:K^b(\sbim(\hf, W))\to  K^b(\sbim(\hf, W_I))$ be the left and right adjoints to the inclusion functor $K^b(\sbim(\hf, W_I))\to K^b(\sbim(\hf, W))$. Then 
    $$\pi_R(X^\bullet)\simeq \pi_L( \ft_{S/I}\otimes_R X^\bullet )\in K^b(\sbim(\hf,W_I))$$
    for all $X^\bullet\in K^b(\sbim(\hf,W))$. 
\end{conjecture}

In the case $W=S_n$, with $\hf=\Cb^n$ and $W_I=S_{r}\times (S_1)^{n-r}$, this conjecture was established in \cite{GHMN19}. For Weyl Groups and their associated Cartan realizations, this was recently shown by Ho and Li in \cite{HL25}.

\subsection{Main Results}

In this paper we prove \cref{ghmnconj} for $W=I_2(m)=\abrac{s,t|\, s^2=t^2=(st)^m=e}$, the dihedral group, for any symmetric realization of $I_2(m)$ satisfying the non-degeneracy conditions stated in \cref{asump1} and \cref{asump2}. This furnishes the first non-crystallographic instance of the phenomenon. As a consequence, we show that tensoring with the complex $\ft_{\set{s,t}}$ defines a Serre functor on $K^b(\sbim(\hf, I_2(m)))$; namely

\begin{theorem}[Serre Duality]
\label{maintheorem}
    For any two complexes $X^\bullet, Y^\bullet \in K^b(\sbim(\hf, I_2(m)))$ one has the following homotopy equivalences of complexes right (or left) $R-$modules.
\[
\uhom_{R^e}(X^\bullet,Y^\bullet) \simeq \uhom_{R^e}(Y^\bullet, \mathrm{FT}_{\set{s,t}}^{-1} \otimes_R X^\bullet)^\vee \simeq \uhom_{R^e}(\mathrm{FT}_{\set{s,t}} \otimes_R Y^\bullet, X^\bullet)^\vee 
\]
\end{theorem}
where $\uhom_{R^e}(-,-)$ is the complex of Homs and $\vee=\Db$\footnote{This equality holds for $R-$modules, but not for arbitrary $R-$bimodules.} is the $R-$linear dual, see \cref{defd}. 

\begin{remark}
    Serre functors for $\kb-$linear categories $\Cc$ typically use the $\kb-$linear dual $(-)^*=\hom_\kb(-, \kb)$ but this requires that the $\hom$ spaces in $\Cc$ are finite $\kb$ modules which is not the case for Soergel Bimodules.  
\end{remark}

\begin{remark}
The category of Soergel modules $\mathbb{S}\mathrm{Mod}(\hf, W)$ is obtained from $\sbim(\hf, W)$ by quotienting out the right (or left) action of $R$. Its objects are of the form
\[ \overline{M}= M\otimes_R R/R^+\cong M\otimes_R \kb \qquad M\in  \sbim(\hf, W)  \]
When $W$ is a Weyl group with Lie algebra $\gf$ and $\hf_{\mathrm{Cart}}$ denotes the Cartan realization, Soergel \cite{Soe90} established the following equivalence of categories $$D^b(\Oc_0)\cong K^b( \mathbb{S}\mathrm{Mod}(\hf_{\mathrm{Cart}}, W))$$
where $\Oc_0$ is the principal block of Category $\Oc$ for $\gf$.  In this setting, $\hom$ spaces in are finite $\kb$ modules and it was shown in \cite{bbm04} and \cite{MS08} that the Serre functor for $D^b(\Oc_0)$ is given by the action of the full twist. \\

Thus, applying the functor $\otimes_R \kb$ to \cref{maintheorem}, we obtain a version of Serre duality for dihedral Category $\Oc$
\[
\uhom_{C}(\overline{X^\bullet},\overline{Y^\bullet}) = \uhom_{C}(\overline{Y^\bullet}, \mathrm{FT}_{\set{s,t}}^{-1} \otimes_R \overline{X^\bullet})^*  \qquad \forall \overline{X^\bullet}, \overline{Y^\bullet}\in K^b( \mathbb{S}\mathrm{Mod}(\hf, I_2(m)))
\] 
where $C=R/(R^W_+)$ is the coinvariant algebra. See \cite[Section 6.4]{GHMN19} for more details. 
\end{remark}

A key step in our proof of \cref{ghmnconj} for dihedral groups is the explicit computation of the left and right adjoints to the inclusion functor $\sbim(\hf, \abrac{t} )\hookrightarrow \sbim(\hf, I_2(m)) $ on objects in $\sbim(\hf, I_2(m))$. In particular, using tools from the diagrammatics of Ext groups between Soergel bimodules developed in \cite{Li25}, we show that

\begin{theorem}
   Let $w\in I_2(m)$ and let $B_w\in \sbim(\hf, I_2(m))$ be the corresponding indecomposable bimodule. Then 
   \begin{enumerate}[(a)]
        \item $\pi^{\pm}_s(R)\cong R$.
       \item $\pi^{\pm}_s(B_s)\cong R(\pm 1)$.
       \item If $w$ contains at least one $t$, $\pi_s^-(B_w)\cong B_t (-\ell(w)+1) $ and $\pi_s^+(B_w)\cong B_t (\ell(w)-1) $.
   \end{enumerate}
\end{theorem}
where $\pi^-_s(M), \pi^+_s(M)$ are defined in \cref{pidef} and shown to be the right and left adjoints in \cref{adjcor}. Unlike the situation in type $A$, where \cite{GHMN19} bypassed such a calculation by appealing to special structural features, the dihedral case requires this explicit analysis. (see \cref{gensect} for more details). Moreover, for $m=3$, this gives explicit descriptions of the partial Hochschild cohomology functors used in the construction of triply-graded link homology. As a consequence, we determine the triply graded link homology of the Whitehead link (see \cref{fig:whitehead}):

\begin{theorem}
For the braid representative $\sigma_1^{-2}\sigma_2\sigma_1^{-1}\sigma_2$ of the Whitehead link,the Poincare series for $\overline{\hhh}$ will be
\begin{align*}
    \overline{\mathscr{P}(A, Q, T)}&=TQ^{-1}+T^2Q^{-3}+A\paren{T^{-1}Q^{-1} + Q^{-3}+\frac{Q^{-3}}{(1-Q^2)} +TQ^{-5}+T^2Q^{-7} } +A^2 \paren{ T^{-1}Q^{-5}+\frac{Q^{-7}}{(1-Q^2)}} 
\end{align*}
\end{theorem}
While computations in triply-graded link homology are generally formidable, the case of positive and negative braid links admits more accessible calculations. For example, $\overline{\hhh}$ of positive torus links $T(m,n)$ with $m,n \geq 0$ was determined in \cite{EH19}, \cite{HM19}, and partial results for all positive and negative braid links were obtained in \cite{last3li}. In contrast, our work provides the first genuinely nontrivial computations of triply-graded link homology in the setting of mixed braid links.

\begin{figure}[H]
    \centering
    \includegraphics[angle=90,origin=c, scale=0.1]{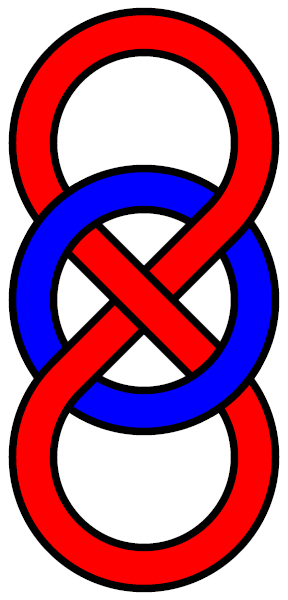}
    \vspace{-4ex}
    \caption{A planar diagram of the Whitehead link.}
    \label{fig:whitehead}
\end{figure}

   \vspace{-5ex}
\subsection{Outline}

\begin{itemize}
    \item In \cref{sect2} we recall the definition of a symmetric realization $\hf$ of a Coxeter system and review the tools needed from \cite{Li25} and \cite{GHMN19}.
    \item \cref{sect3} applies the tools from \cite{Li25} to derive certain vanishing results needed in \cref{sect4}.
    \item \cref{sect4} contains the proof of Relative Serre Duality, adapting the approach developed in \cite{GHMN19}. Our proof of Serre duality differs slightly, as \cite[Lemma 3.9]{GHMN19} is not correct as written. We provide an alternative argument, from which the main results of \cite{GHMN19} still follow.  
    \item In \cref{gensect} we conjecture an explicit description of the adjoints of the inclusion functor for general parabolic subgroups of Coxeter groups, and lightly sketch how \cref{ghmnconj} could be deduced from one or two central statements. 
    \item In \cref{whiteheadsect}, we compute the triply-graded link homology of the Whitehead link (Rolfsen $L5a1$) using tools from \cref{sect3}. This section may be read independently of \cref{sect4,gensect}. 
\end{itemize}

\section{Preliminaries}
\label{sect2} 

\begin{definition}
Let $\kb$ be a commutative ring. A symmetric realization of $(W,S)$ over $\kb$ is a triple $$\hf=(V, \set{\alpha_s}_{s\in S}\subset V^*, \set{\alpha_s^\vee}_{s\in S}\subset V)$$ where $V$ is a free, finite rank $\kb$ module such that for $\abrac{-,-}$ the natural pairing between $V, V^*$,                                                                                                                                                                                                                                             
\begin{enumerate}[(1)]
    \item $\abrac{\alpha_s^\vee, \alpha_s}=2$ for all $s\in S$.
    \item $a_{st}=a_{ts}$.
    \item The assignment $s(v)=v-\abrac{v, \alpha_s} \alpha_s^\vee$ for all $v\in V$ yields a representation of $W$.
    \item For each $s,t\in S$ let $a_{st}:=\abrac{\alpha_s^\vee, \alpha_t}$ and $m_{st}$ the order of $st$ in $W$. Write $a_{st}=-(q+q^{-1})$ for $q\in \kb$. Then $[m_{st}]_{q}=0$ where $[k]_q$ are the quantum numbers defined in Section 3.1 of \cite{DC}.
\end{enumerate}
\end{definition}

\begin{definition}
    A realization $\hf$ has fundamental weights if $\forall s\in S$, $\exists \rho_s\in V^*$ s.t. $\abrac{\rho_s, \alpha_t^\vee}=\delta_{st}$ for all coroots $\alpha_t^\vee$.
\end{definition}

It is clear from the definition that fundamental weights for $\hf$ exists $\iff \set{\alpha_s^\vee}$ are linearly independent. Throughout the remainder of the paper, until \cref{gensect}, we will work with the Coxeter system $(W=I_2(m), S=\set{s,t})$ and assume that our realization $\hf$ of $I_2(m)$ satisfies the following

\begin{assumption}
\label{asump1}
     $\hf$  has fundamental weights, $\mathrm{char}\ \kb=0$, and $\set{\alpha_s, \alpha_t}$ is linearly independent.
\end{assumption}

\begin{assumption}
\label{asump2}
    For all $k<m$, $[k]$ is invertible over $\kb$ and $q$ is a primitive $2m$-th root of unity. 
\end{assumption}

These assumptions are needed so that the results of \cite{Li25} hold.

\subsection{Notation}

\begin{itemize}
    \item $\boxed{M}$ around an object $M$ indicates that $M$ is placed in cohomological degree 0. 
    \item $(1)$ shifts the internal grading down by 1, e.g. $M(1)_d=M_{d+1}$. 
    \item $\bs(\underline{w})=B_{s_1}\otimes \ldots \otimes B_{s_k}$ for an expression $\underline{w}=(s_1\ldots, s_k)$, $s_i\in S$.  
    \item $B_w$ is the indecomposable Soergel Bimodule corresponding to $w\in W$. 
    \item Given an $R-$bimodule $M$, $\hh^k(M)=\ext_{R^e}^k(R, M)$ and $\hh_k(M)=\mathrm{Tor}^{R^e}_k(R, M)$.
    \item $\abrac{X_i^\bullet}_{i\in I}=$the span of $\set{X_i^\bullet}_{i\in I}$ which is the smallest graded triangulated subcategory containing the complexes $\set{X_i^\bullet}_{i\in I}$. Equivalently, close under taking $(1)$ and $[1]$ shifts, and taking cones and direct sums.  
    \item $\mathrm{Tot}$ is the total complex of a double complex. 
\end{itemize}

\subsection{Review of \cite{Li25}}

\begin{definition}
\label{defd}
Let $\Db: R^e-\mathrm{gmod}\to R^e-\mathrm{gmod}$ be the functor
\[ \Db(N):=\hom_{\_ R}(N, R) \]
where $\_ R$ means we take right $R$ module homomorphisms. This has a $R^e-\mathrm{gmod}$ structure defined as $(r\cdot f \cdot r^\prime)(b):=r \cdot f(b) \cdot r^\prime$.
\end{definition}

\begin{lemma}[{\cite[Theorem 7.2]{Li25}, \cite[Theorem 3.12]{Li25}}]
\label{li72}
Let $\rho_s^e:=\rho_s\otimes 1- 1\otimes \rho_s\in R^e$.  If \cref{asump1} and \cref{asump2} holds for the realization $\hf$, then we have isomorphism of bimodules
    \[\ker \rho_s^e( \bs(\un{w})) \overset{R^e}{\cong} \hom_{R^e}(B_t, \bs(\un{w}))(1) \qquad  \mathrm{coker} \rho_s^e(\bs(\un{w}))  \overset{R^e}{\cong} \Db(\hom_{R^e}(B_t, \bs(\un{w})(1)) \qquad \forall \un{w}\neq \emptyset\]
\end{lemma}

\begin{warn}
 It is tempting to use the adjunction $\hom_{R^e}(B_t, \bs(\un{w}))\cong \hom_{R^e}(R,B_t\otimes_R \bs(\un{w})) $ above but this adjunction is only an isomorphism of right $R-$modules, not as bimodules.  
\end{warn}

\begin{lemma}
\label{li36}
  Given $B\in \sbim(\hf  ,I_2(m))$, if \cref{asump1} and \cref{asump2} holds for the realization $\hf$, then  $\ext^{k}_{R^e}(R, B)$ is a free $R-$module for all $k$. 
\end{lemma}
\begin{proof}
    \cite[Theorem 7.12]{Li25} shows this for the geometric realization and \cite[Theorem 3.6]{Li25} shows that $\ext$ for other realizations can be gotten by tensoring over $\kb$ with a finite free module over $\kb$. 
\end{proof}

\subsection{Review of \cite{GHMN19}}

In this subsection, we collect some definitions and results from \cite{GHMN19} that are used later in this paper. 

\begin{definition}
Suppose $\Ac$ is a triangulated category that has a monoidal structure $\otimes$. Then $\Ac$ is a triangulated monoidal category with monoidal unit $\mathbbm{1}$ if $a\otimes -$ and $-\otimes a$ are triangulated functors for any $a\in \Ac$. 
\end{definition}

\begin{definition}
A \textit{unital idempotent} in $\Ac$ is a pair $(\mathbf{Q}, \eta)$ where $\mathbf{Q}\in \Ac$ and $\eta: \mathbbm{1}\to \mathbf{Q}$, s.t. $\eta \otimes 1_{\mathbf{Q} }, 1_{\mathbf{Q}}\otimes \eta: \mathbf{Q}\to \mathbf{Q}\mathbf{Q}$ are isomorphisms. A \textit{counital idempotent} in $\Ac$ is a pair $(\mathbf{P}, \epsilon)$ where $\mathbf{P}\in \Ac$ and $\eta:\mathbf{P}\to \mathbbm{1}$, s.t. $\epsilon \otimes 1_{\mathbf{P} }, 1_{\mathbf{P}}\otimes \epsilon: \mathbf{P} \mathbf{P}\to \mathbf{P}$ are isomorphisms.
\end{definition}

\begin{definition}
Given a triangulated monoidal category $\Ac$, an $A-$module category $\Mc$ is a triangulated category such that the action of $\Ac$ on $\Mc$ is triangulated.
\end{definition}

\begin{example}
    If $\Cc$ is an additive category, then $\Ac=K^b( \mathrm{End}(\Cc))$ is a triangulated monoidal category where $\mathrm{End}(\Cc)$ is the category of endofunctors of $\Cc$. Moreover, $\Mc=K^b(\Cc)$ is an $\Ac-$module category, where given $M^\bullet\in K^b(\Cc)$ and $F= [0\to \boxed{F}\to 0]\in \Ac$ the action of the endofunctor $F$ on $M^\bullet$ is given by
\[F\cdot M^\bullet= [\ldots \to F(M^i)\xrightarrow{F(d)} F(M^{i+1})\to \ldots ]\]
For general $F^\bullet\in \Ac$, we take the total complex for $F^\bullet\cdot M^\bullet$.
\end{example}

\begin{definition}[Two-Step]
Let $\Mc_1, \Mc_2\subset \Mc$ be full triangulated subcategories. Then $(\Mc_1, \Mc_2)$ is a a semi-orthogonal decomposition of $\Mc$, written $\Mc\simeq (\Mc_1\to \Mc_2)$ if 
\begin{enumerate}[(1)]
\item Each object $X$ of $\Mc$ fits into a distinguished triangle
\[ X_2\to X \to X_1\xrightarrow{+1} \qquad X_i\in \Mc_i \]
 \item $\hom_{\Mc}(\Mc_2, \Mc_1)=0$
\end{enumerate}
\end{definition}

\begin{lemma}[{\cite[Corollary 4.14]{GHMN19}}]
\label{ghmn414}
    If $\mathbf{P}$ (resp.\ $\mathbf{Q}$) is a counital (resp.\ unital) idempotent in $\mathcal{A}$ then each $\mathcal{A}$-module category $\mathcal{M}$ inherits a semi-orthogonal decomposition $\mathcal{M} \simeq (\ker \mathbf{P} \to \mathrm{im}\, \mathbf{P})$ (resp.\ $\mathcal{M} \simeq (\mathrm{im}\, \mathbf{Q} \to \ker \mathbf{Q})$).
\end{lemma}

\begin{lemma}[{\cite[Lemma 4.17]{GHMN19}}]
\label{ghmn417}
    Let $\{X_i\}_{i\in I}$ and $\{Y_j\}_{j\in J}$ be a collection of objects in $\mathcal{M}$. Let $\mathbf{E}$ be a unital or counital idempotent in $\mathcal{A}$. Assume that $\mathbf{E}(X_i) \simeq 0$ for all $i$ and $\mathbf{E}(Y_j) \simeq Y_j$ for all $j$. Then:
\begin{itemize}
    \item[(3)] If $\{X_i\}_{i\in I}$ and $\{Y_j\}_{j\in J}$ together span all of $\mathcal{M}$ as a triangulated category, then $\{X_i\}_{i\in I}$ spans the kernel of $\, \mathbf{E}$ while $\{Y_j\}_{j\in J}$ spans the image of $\, \mathbf{E}$.
\end{itemize}
\end{lemma}

\section{Partial Trace}
\label{sect3}

\begin{definition}
\label{pidef}
    Given $s\in S$, $M\in \sbim(\hf  ,I_2(m))$, let $\pi^-_s$, $\pi^+_s$ be the kernel, cokernel, respectively, of the action of $\rho_s^e=\rho_s\otimes 1- 1\otimes \rho_s$ on $M$. 
    \[ 0\to \pi^-_s(M)\to M\xrightarrow{\rho_s\otimes 1- 1\otimes \rho_s}M \to  \pi^+_s(M)\to 0 \]
\end{definition}

\begin{remark}
    Let $\eta$ be the inclusion map $ \pi^-_s(M)\to M$ above and $\epsilon$ the projection map $M \to  \pi^+_s(M)$. Then $(\pi^-_s, \eta)$, $(\pi^+_s, \epsilon)$  are counital/unital idempotents in $K^b(\mathrm{End}(\sbim(\hf  ,I_2(m))  ))$ respectively. 
\end{remark}

\begin{lemma}
\label{lem:jker}
    Let $J$ be the projector to a sub-bimodule $N\subset M$ and let $T: M\to M$ be a bimodule endomorphism. If $TJ=JT$ then $\ker(T|_N)=J(\ker T)$. 
\end{lemma}
\begin{proof}
  We always have $\ker(T|_N)\subset J(\ker T)$ and $TJ=JT$ implies $J(\ker T)\subset \ker(T|_N)$. 
\end{proof}

\begin{theorem}
\label{partialprop}
   Let $w\in I_2(m)$ and $B_w\in \sbim(\hf, I_2(m))$. Then
   \begin{enumerate}[(a)]
        \item $\pi^{\pm}_s(R)\cong R$.
       \item $\pi^{\pm}_s(B_s)\cong R(\pm 1)$.
       \item If $w$ contains at least one $t$, $\pi_s^-(B_w)\cong B_t (-\ell(w)+1) $ and $\pi_s^+(B_w)\cong B_t (\ell(w)-1) $.
   \end{enumerate}
\end{theorem}
\begin{proof}
    Let $\underline{w}$ be a reduced expression. From \cref{li72} we have that 
    \begin{equation}
    \label{eq1}
        \pi^-_s(\bs(\un{w})) \overset{R^e}{\cong} \hom_{R^e}(B_t, \bs(\un{w}))(1)
    \end{equation}
   Let $\jw_{\un{w}}$ be the projector to the indecomposable $B_w$ in $\bs(\un{w})$. It is clear diagrammatically that $\jw_{\un{w}}$ commutes with $\rho_s^e$ and thus
    \begin{equation}
    \label{eq2}
        \pi_s^-(B_w)=\ker \rho_s^e(\un{w})|_{B_w}\xequals{\cref{lem:jker}}\jw_{\un{w}}\paren{\ker \rho_s^e(\un{w})}\xequals{\cref{eq1}}\jw_{\un{w}} \paren{\hom_{R^e}(B_t, \bs(\un{w}))(1)}= \hom_{R^e}(B_t, B_w)(1)
    \end{equation}  
    $(b)$ Using \cref{eq2} and the Double Leaves basis \cite[Chapter 10.4]{EMTW20}, we see that
    \[ \pi_s^-(B_s)=\hom_{R^e}(B_t, B_s)(1)=\raisebox{-2ex}{\begin{tikzpicture}[xscale=0.25,yscale=0.3,thick,baseline]
 \draw[blue] (0,0) -- (0,1);
 \draw[red] (0,2) -- (0,3);
 \node[dot, blue] at (0,1) {};
  \node[dot, red] at (0,2) {};
\end{tikzpicture}} \  R(1) \overset{R^e}{\simeq} R(-1) \]
$(c)$ Let \rect denote the Jones-Wenzl projector $\jw_{\un{w}}$ where WLOG $w$ starts with a $t$. Because $\jw_{\un{w}}$ is annihilated by any generalized pitchfork, using \cref{eq2} we compute that
\[ \pi_s^-(B_w)=\hom_{R^e}(B_t, B_w)(1)=\stackon[1pt]{\bcounit}{\stackon[-1pt]{\hspace{-0.5ex}\bunit \runit \raisebox{4pt}{$\cdots$}}{\rect}} \  R(1) \oplus \raisebox{-1.8ex}{\stackon[-1pt]{\hspace{0.5ex}\begin{tikzpicture}[xscale=0.25,yscale=0.3,thick,baseline]
 \draw[blue] (0.9,0) -- (0.9,2.5);
 \draw[red] (1.5,1.5) -- (1.5,2.5);
 \node at (2.6,2) {$\ldots$};
  \node[dot, red] at (1.5,1.5) {};
\end{tikzpicture}}{\rect}} \ R(1) \overset{R^e}{\simeq} B_t(-\ell(w)+1) \]
where the $\ldots$ are all dot morphisms of the appropriate color and where we have used the right $R-$module isomorphism $B_t\cong \stackon{\bcounit}{\bunit} R(1)\oplus \bzero R(1)$ (we are identifying the morphisms on the RHS with elements of $B_t$ by applying to $1\otimes_t 1$). To prove the analogous statements for $\pi_s^+$, note that \cref{li72} tells us for $\underline{w}\neq \emptyset$
\begin{equation}
     \pi^+_s(\bs(\un{w})) \overset{R^e}{\cong} \Db(\pi_s^-(\bs(\un{w}) ) )
\end{equation}
Using our results for $\pi_s^-$ and \cite[Proposition 18.9]{EMTW20} then gives the analogous results for $\pi_s^+$.
\end{proof}

\begin{corollary}
\label{adjcor}
    $\pi^-_s, \pi^+_s$ are the right and left adjoints to the inclusion $\sbim(\hf, \abrac{t} )\hookrightarrow \sbim(\hf, I_2(m)) $.
\end{corollary}
\begin{proof}
    Similar to the proof of \cite[Proposition 5.11]{GHMN19}\footnote{There is a typo in the proof, $\pi^-$ is a co-localization functor, $\pi^+$ is a localization functor.}, $\pi^-_s, \pi^+_s$ are co-localization and localization functors so by \cite[Lemma 4.4]{GHMN19} it suffices to show the essential image of $\pi^-_s, \pi^+_s$ is $\sbim(\hf, \abrac{t} )$. This follows from the above lemma.
\end{proof}

\begin{corollary}
     $\pi_s^-:\sbim(\hf, I_2(m))\to \sbim(\hf, \abrac{t} ) $ is represented by the object $B_t(-1)$.
\end{corollary}
\begin{proof}
    Direct consequence of \cref{eq2}.
\end{proof}

\begin{remark}
   Because both $\pi^-_s$, $\pi^+_s$ are additive, they descend to triangulated functors $K^b(\sbim(\hf, I_2(m)))\to K^b(\sbim(\hf, \abrac{t} )) $.
\end{remark}

\begin{lemma}
\label{lembilinear}
    $\pi^{\pm}_s$ is $\sbim(\hf, \abrac{t} )-$bilinear, meaning
    \[ \pi^\pm_s (M\otimes_R X \otimes_R N)=M\otimes_R\pi^\pm_s(X) \otimes_R N  \]
    for all $M, N\in \sbim(\hf, \abrac{t} )$ and $X\in \sbim(\hf, I_2(m))$.
\end{lemma}
\begin{proof}
    This is because $\abrac{\rho_s, \alpha_t^\vee}=0$ so $\rho_s$ slides past $M, N$. 
\end{proof}

\begin{lemma}
\label{deltabchange}
In $I_2(m)$, $\displaystyle \delta_w=\sum_{y\le w}(-v)^{\ell(w)-\ell(y)}b_y$.
\end{lemma}
\begin{proof}
    It is well known that for the dihedral group, ``all KL polynomials are 1" meaning that for $w\in I_2(m)$
    \[  b_{w}=\sum_{y\le w} v^{\ell(w)-\ell(y) } \delta_y \implies  v^{-\ell(w)}b_{w}=\sum_{y\le w} v^{-\ell(y) } \delta_y  \]
    The RHS above is suitable for Mobius inversion. Because the Mobius function for the Bruhat poset of a Coxeter system is of the form $\mu(y,w)=(-1)^{\ell(w)-\ell(y)}$ \cite{Stem07} it follows that
    \[  v^{-\ell(w)}\delta_w=\sum_{y\le w} \mu(y,w)v^{-\ell(y)}b_y\implies  \delta_w=\sum_{y\le w}(-v)^{\ell(w)-\ell(y)}b_y \]
\end{proof}

\begin{prop}
    The minimal complex of $F^\bullet_{(st)^k}$ in $K^b(\sbim(\hf  ,I_2(m)))$ when $1\le k\le \lfloor\frac{m}{2}\rfloor$ is of the form (boxed term is in cohomological degree 0)
    \begin{equation}
    \label{mineq}
        \boxed{B_{(st)^k}} \xrightarrow{} B_{(st)^{k-1}s}(1) \oplus B_{(ts)^{k-1}t}(1) \xrightarrow{} B_{(st)^{k-1}}(2) \oplus B_{(ts)^{k-1}}(2)\to \ldots \to R(2k)   
    \end{equation} 
    where the differentials are uniquely determined up to a scalar. Note that besides the first and last cohomological degrees, the rest have exactly two summands in them. A similar result holds for $F^\bullet_{(st)^ks}$.
\end{prop}
\begin{proof}
    $(st)^k$ is a reduced expression in $I_2(m)$ for $k\le \lfloor\frac{m}{2}\rfloor$ and thus by \cite[Theorem 19.47]{EMTW20} $F^\bullet_{(st)^k}$ is perverse, meaning that objects with internal degree shifts $(n)$ in the minimal complex for $F^\bullet_{(st)^k}$ are in cohomological degree $n$. It follows that the class of $F^\bullet_{(st)^k}$ in the Grothendieck group determines the terms for the minimal complex for $F^\bullet_{(st)^k}$ so by \cref{deltabchange} we obtain \cref{mineq}, modulo the differentials. However, for $w,u\in I_2(m)$ s.t. $\ell(w)=\ell(u)+1$, we claim
    \[ \dim_\kb \hom^1( B_w, B_u )=1 \]
    This will follow from Soergel's Hom formula as we compute that
    $$\epsilon(b_wb_u)=\sum_{x\le w, \, y\le u}v^{\ell(w)-\ell(x)+\ell(u)-\ell(y)} \epsilon( \delta_x \delta_y)$$
    Using \cite[Lemma 3.15]{EMTW20} we have $x=y^{-1}$, and thus the RHS equals $\sum_{y\le v}v^{2\ell(u)-2\ell(y)+1}  $ and it follows the coefficient of $v^1$ in $\epsilon(b_wb_u)$ is 1 as desired.
\end{proof}
\begin{remark}
    $\dim_\kb \hom^1( B_w, B_v )=1 $ when $v<w$ and $\ell(w)=\ell(v)+1$ actually holds in any Coxeter group. Using \cite[Lemma 3.26]{EMTW20}, it follows that 
    $$  b_w=\delta_w+\sum_{\substack{ \ell(w)=\ell(y)+1 \\ y<w}}v\delta_y+\sum_{\substack{ \ell(w)>\ell(z)+1 \\ z<w}} h_{z,w} \delta_{z}, \qquad h_{z,w}\in v\Zb[v]$$
    Thus when computing $\epsilon(b_wb_u)$, \cite[Lemma 3.15]{EMTW20} again says we must pair $x$ with its inverse and so only $\delta_u$, the leading term of $b_u$ can contribute a factor of $v$.
\end{remark}

\begin{corollary}
\label{pis0}
\begin{enumerate}[(i)]
    \item  For $1\le k\le \lfloor\frac{m}{2}\rfloor$,  $\pi^+_s\paren{F^\bullet_{(st)^k}}\simeq \pi^+_s\paren{F^\bullet_{(st)^ks}}\simeq 0$ in $\sbim(\hf, \abrac{t} )$.
    \item $\pi^+_s\paren{F_s}\simeq 0$
\end{enumerate}
\end{corollary}
\begin{proof}
    $(i)$ It suffices to show $\pi^+_s\paren{F^\bullet_{(st)^k}}\simeq 0$ as the other case proceeds similarly. Applying $\pi^+_s$ to \cref{mineq} and using \cref{partialprop}, $(c)$ we obtain 
    \[\pi^+_s\paren{F^\bullet_{(st)^k}}=\boxed{B_t(2k-1)}\to B_t(2k-1)\oplus B_t(2k-1)\to \ldots \to B_t(2k-1)\oplus R(2k) \to R(2k) \]
    As $\hom_{R^e}(B_t, B_t)\cong R\oplus R(-2)$, the only degree 0 morphisms are isomorphisms and thus every component of the differentials above are isomorphisms except in the last 3 terms on the right. We can then use Gaussian elimination in the homotopy category starting from the leftmost term.
    \[\pi^+_s\paren{F^\bullet_{(st)^k}}\simeq\boxed{\xcancel{B_t(2k-1)}}\to \xcancel{B_t(2k-1)}\oplus B_t(2k-1)\to \ldots \to B_t(2k-1)\oplus R(2k) \to R(2k) \]
    This continues all the way to the right where we end up canceling $R(2k)\to R(2k)$. $(ii)$ proceeds similarly.
\end{proof}

\begin{corollary}
\label{negpis0}
\begin{enumerate}[(i)]
    \item  For $1\le k\le \lfloor\frac{m}{2}\rfloor$,  $\pi^-_s\paren{F^\bullet_{(st)^{-k}}}\simeq \pi^-_s\paren{F^\bullet_{(st)^{-k}s^{-1}}}\simeq 0$ in $\sbim(\hf, \abrac{t} )$.
    \item $\pi^-_s\paren{F_{s^{-1}}}\simeq 0$
\end{enumerate}
\end{corollary}
\begin{proof}
    Same proof as above, but where we use that the minimal complex for $F^\bullet_{(st)^{-k}}$ is obtained from $F^\bullet_{(st)^{k}}$ by flipping all arrows and reversing the internal gradings. Concretely,
    \begin{equation}
    \label{negmineq}
       F^\bullet_{(st)^{-k}}\simeq R(-2k)\to \ldots  \xrightarrow{} B_{(st)^{k-1}s}(-1) \oplus B_{(ts)^{k-1}t}(-1) \xrightarrow{}  \boxed{B_{(st)^k}}
    \end{equation}
\end{proof}

\section{Serre Duality}
\label{sect4}

Similar to \cite[Section 6]{GHMN19}, a key role in our proof will be played by ``link-splitting" maps\footnote{The name comes from the fact that any link can be unlinked by appropriately switching positive and negative crossings.}. The Rouquier complex $F_s$ has a chain map $\psi_s$ to $F_s^{-1}$ as seen below.
\[\begin{tikzcd}
F_s \arrow[d, "\psi_s"] & = & & \boxed{B_s} \arrow[d, "\id"]\arrow[r] & R(1) \\
F_s^{-1} & = & R(-1) \arrow[r] & \boxed{B_s}&
\end{tikzcd}
\]
After Gaussian elimination, $\cone(\psi_s)\simeq R(-1)\xrightarrow{\alpha_s}\boxed{R(1)}$ and so we have the following distinguished triangle
\begin{equation}
\label{spliteq}
    F_s^2\xrightarrow{\id_{F_s}\otimes_R \psi_s} R\to \mathrm{Tot}\paren{F_s(-1)\xrightarrow{\alpha_s} \boxed{F_s(1)}} \xrightarrow{+1}
\end{equation}

\begin{prop}
\label{preprop}
Let $\overbrace{F_sF_t\ldots}^{k }$ be the Rouquier complex where we alternate between $F_s$ and $F_t$ for $k$ terms.
\begin{enumerate}[(i)]
    \item $\pi_s^+(\overbrace{F_s\ldots F_t}^{a }F_s^2\overbrace{F_t\ldots}^{b })\simeq \pi_s^+(\overbrace{F_s\ldots F_s}^{c }F_t^2\overbrace{F_s\ldots}^{d })\simeq 0$ for $|a-b|\ge 2$ and $|c- d|\ge 2$ and $a+b\le 2m-3$
    \item $\pi_s^+(\overbrace{F_sF_t\ldots}^{k })\simeq 0$ for all $2\le k \le 2m-3$.
\end{enumerate}
\end{prop}
\begin{proof}
Suppose that $(ii)$ is true up to $k=k_0$. As we are in a triangulated monoidal category, from \cref{spliteq} we obtain the distinguished triangle
\[ \overbrace{F_s\ldots F_t}^{a }F_s^2\overbrace{F_t\ldots}^{b }\to \overbrace{F_s\ldots F_s}^{a -1}F_t^2\overbrace{F_s\ldots}^{b -1}\to  \mathrm{Tot}\paren{\overbrace{F_sF_t\ldots}^{a+b+1 }(-1)\to\overbrace{F_sF_t\ldots}^{a+b+1 }(1) } \]
If $a+b\le k_0-1$, then applying $\pi_s^+$  the last term will be 0 by assumption. Repeating this, we obtain the chain of homotopy equivalences
\[ \pi_s^+\paren{  \overbrace{F_s\ldots F_t}^{a }F_s^2\overbrace{F_t\ldots}^{b }}\simeq \pi_s^+\paren{  \overbrace{F_s\ldots F_s}^{a-1 }F_t^2\overbrace{F_s\ldots}^{b -1}}\simeq \ldots \pi_s^+\paren{\overbrace{F_s\ldots}^{|a-b| }}  \]
As $2\le |a-b|\le k_0$ the last term is homotopy equivalent to 0. Thus,  $(ii)$ up to $k_0\implies (i)$ up to $a+b\le k_0-1$. \\

Now suppose that $(i)$ is true up to $a+b\le k_0-1$. Note $(ii)$ is already true for $k\le m$ by \cref{pis0}. When $k_0+1>m$ we can apply a braid relation to see that when $m$ is odd that
    \[ \overbrace{F_sF_t\ldots F_s}^{k_0+1 }=\overbrace{F_s\ldots F_s}^{m}\overbrace{F_t\ldots }^{k_0+1 -m}\xequals{braid}\overbrace{F_t\ldots F_s}^{m-1}F_t^2\overbrace{F_s\ldots F_t}^{k_0-m}   \]
    If $m$ is even then we have $F_s^2$ instead of $F_t^2$ in the middle. Since $m-1+(k_0-m)=k_0-1$ and $|k_0-m-(m-1)|=|k_0-2m+1|\ge 2$ as $k_0\le 2m-3$, $(i)$  up to $a+b\le k_0-1$ shows the RHS above is homotopy equivalent 0. Thus, $(i)$ up to $a+b\le k_0-1$ implies $(ii)$ up to $k_0+1$.
\end{proof}

\subsection{Semi-Orthogonal Decomposition}

\begin{lemma}
Given $w\in I_2(m)$, let $F_w$ and $F_w^{-1}$ be the Rouquier complexes for the positive/negative braid lifts of a chosen reduced expression for $w$. Then $\set{F_w}_{w\in I_2(m)}$ or $\set{ F_w^{-1}}_{w\in I_2(m)}$ spans $K^b(\sbim(\hf, I_2(m)))$.
\end{lemma}
\begin{proof}
The indecomposable bimodules $\set{B_w}_{w\in I_2(m)}$ clearly span $K^b(\sbim(\hf, I_2(m)))$. Taking the cone of the map
\begin{equation}
\label{inceq}
\begin{tikzcd}
 R(1) \arrow[d, ]&= & & R(1) \arrow[d, "\id"] \\
 F_s &= &\boxed{B_s} \arrow[r] & R(1)
\end{tikzcd}
\end{equation}
shows $B_s$ (and $B_t$) are in the span of $\set{F_w}_{w\in I_2(m)}$. Now proceed by induction on the length of $w$. Using \cref{mineq} we have an inclusion map of complexes $F^\bullet_{(st)^{k-1}s}(1)[-1]\to F^\bullet_{(st)^{k}} $
\[\begin{tikzcd}
 F^\bullet_{(st)^{k-1}s}(1) \arrow[d, ]&\simeq & & B_{(st)^{k-1}s}(1) \arrow[d, "\id"] \ar[r]& \ldots \ar[d, "\id"] \\
F^\bullet_{(st)^{k}} &\simeq &\boxed{B_{(st)^{k}}} \arrow[r] & B_{(ts)^{k-1}t}(1)\oplus B_{(st)^{k-1}s}(1) \ar[r]& \ldots 
\end{tikzcd}\]
The cone is isomorphic to the two term complex $\sbrac{\boxed{B_{(st)^{k}}} \to B_{(ts)^{k-1}t}(1)}$. $B_{(ts)^{k-1}t}$ is in the span of $\set{F_w}_{w\in I_2(m)}$ by induction so taking the cone of the inclusion map (similar to \cref{inceq}) shows $B_{(st)^{k}}$ is also in the span.\\

The proof for $\set{ F_w^{-1}}_{w\in I_2(m)}$ spanning proceeds similarly where we use \cref{negmineq} instead.
\end{proof}

\begin{corollary}
\label{semicor}
    Let \, $\Uc_t^{\pm}$ be the full triangulated subcategories spanned by the Rouquier complexes $\set{F_w}_{e,t\neq w\in I_2(m)}$ and $\set{F_w^{-1}}_{e,t\neq w\in I_2(m)}$, respectively. Then we have semi-orthogonal decompositions
    \begin{align*}
       K^b( \sbim(\hf, I_2(m) )) & \simeq \paren{ \Uc_t^{-} \to K^b( \sbim(\hf, \abrac{t} ))  } \\
       & \simeq \paren{  K^b( \sbim(\hf, \abrac{t} ))\to \Uc_t^{+}   }
    \end{align*}
\end{corollary}
\begin{proof}
 $\mathbf{P}=\pi^-_s$ is a counital idempotent and \cref{partialprop} shows that $\pi^-_s(R)=R, \pi^-_s(B_t)=B_t$ and thus $\pi^-_s(F_t)=F_t$. \cref{negpis0} shows that $\pi^-_s(F^\bullet_w)\simeq 0$ for $e,t\neq w\in I_2(m)$. \cref{ghmn417} then implies 
 that
 \begin{equation}
 \label{kerpieq}
     \ker \pi^-_s=\Uc^-_t, \qquad \im \pi^-_s=K^b( \sbim(\hf, \abrac{t} ))
 \end{equation}
 
 The first semi-orthogonal decomposition now follows from \cref{ghmn414}. Similarly, $\mathbf{Q}=\pi^+_s$ is a unital idempotent and \cref{partialprop} shows $\pi^+_s(F_t)=F_t$ while \cref{negpis0} shows that $\pi^+_s(F^\bullet_w)\simeq 0$ for $e,t\neq w\in I_2(m)$ so \cref{ghmn417} implies that
  \begin{equation}
     \ker \pi^+_s=\Uc^+_t, \qquad \im \pi^+_s=K^b( \sbim(\hf, \abrac{t} ))
 \end{equation}
 and the second semi-orthogonal decomposition also follows from \cref{ghmn414}.
\end{proof}

\begin{lemma}
\label{tenequiv}
   Tensoring on the left (or right) with $\ft_{\set{s,t}/t}$ (see \cref{fteq} below) gives an equivalence of categories $\Uc_t^{-}\to \Uc_t^{+}$ with inverse $\ft_{\set{s,t}/t}^{-1}$. 
\end{lemma}
\begin{proof}
    We first show that tensoring on the left with $\ft_{\set{s,t}/t}$ sends $\Uc_t^{-}$ to $\Uc_t^{+}$ by showing it on the generators of $\Uc_t^{-}$. When $m$ is odd, we have
    \[ \ft_{\set{s,t}/t} \overbrace{F_s^{-1}F_t^{-1}\ldots }^{k}= \overbrace{F_s\ldots F_s}^{m-2}F_t^2\overbrace{F_s\ldots F_s}^{m-2-k}\implies  \text{Applying }\pi_s^+ \simeq 0  \qquad k\le m-2   \]
    \[ \ft_{\set{s,t}/t} \overbrace{F_s^{-1}F_t^{-1}\ldots }^{m-1}= \overbrace{F_s\ldots F_t}^{m-1}\implies  \text{Applying }\pi_s^+ \simeq 0\]
    \[ \ft_{\set{s,t}/t} \overbrace{F_s^{-1}F_t^{-1}\ldots }^{m}= \overbrace{F_s\ldots F_t}^{m-1}F_s^{-1}\xequals{braid} F_t^{-1} \overbrace{F_s\ldots F_s}^{m-2} F_t  \implies  \text{Applying }\pi_s^+ \simeq 0 \]
   where we have used \cref{preprop} in the last steps above. One can check at the level of the braid group that $\ft_{\set{s,t}/t} t^{-1}=t^{-1} \ft_{\set{s,t}/t}$ and thus if our generator starts with a $F_t^{-1}$ on the left
\[ \pi_s^+\paren{\ft_{\set{s,t}/t} \overbrace{F_t^{-1}F_s^{-1}\ldots }^{k}}= \pi_s^+\paren{F_t^{-1} \ft_{\set{s,t}/t} \overbrace{F_s^{-1}\ldots }^{k-1}}= F_t^{-1}\pi_s^+\paren{ \ft_{\set{s,t}/t} \overbrace{F_s^{-1}\ldots }^{k-1}}\simeq 0 \]
   where we have used bilinearity of $\pi_s^+$ and the previous case. Similar calculations hold when $m$ is even. Thus by \cref{kerpieq} the image of $\Uc_t^{-}$ lands in $\Uc_t^{+}$ as desired.\\

   Similarly, tensoring on the left with $\ft_{\set{s,t}/t}^{-1}$ sends $\Uc_t^{+}$ to $\Uc_t^{-}$ and it is clear that these functors are inverse equivalences.
\end{proof}

\subsection{Relative Serre Duality}

\begin{theorem}
\label{relsertheorem}
    $\pi_s^-(X^\bullet)\simeq \pi_s^+( \ft_{\set{s,t}/t}\otimes_R X^\bullet )\simeq \pi_s^+(X^\bullet\otimes_R \ft_{\set{s,t}/t} )$ for all $X^\bullet\in K^b(\sbim(\hf, I_2(m)))$.
\end{theorem}
\begin{proof}
It suffices to prove the first equivalence as the second proceeds similarly. When $m$ is odd, we have that 
   \begin{equation}
   \label{fteq}
       \ft_{\set{s,t}/t}=F^\bullet_{(st)^m}F_t^{-2}\simeq\overbrace{F_s\ldots F_t}^{m-1}\overbrace{F_s\ldots F_s}^{m} F_t^{-1}\xequals{braid}  \overbrace{F_s\ldots F_s}^{m-2}F_t^2\overbrace{F_s\ldots F_s}^{m-2}  
   \end{equation} 
   As in the proof of \cref{preprop} using \cref{spliteq} we obtain the distinguished triangle
   \[ \overbrace{F_s\ldots F_s}^{m-2 }F_t^2\overbrace{F_s\ldots F_s}^{m-2 }\to \overbrace{F_s\ldots F_t}^{m -3}F_s^2\overbrace{F_t\ldots F_s}^{m -3}\to  \mathrm{Tot}\paren{\overbrace{F_sF_t\ldots}^{2m-3}(-1)\to\overbrace{F_sF_t\ldots}^{2m-3 }(1) } \xrightarrow{+1}\]
   Using \cref{preprop} $(ii)$ the last term is homotopy equivalent to 0 after applying $\pi_s^+$, so we obtain the following chain of homotopy equivalences
   \begin{equation}
   \label{pifteq}
    \pi_s^+\paren{  \ft_{\set{s,t}/t}}=  \pi_s^+\paren{  \overbrace{F_s\ldots F_s}^{m-2 }F_t^2\overbrace{F_s\ldots F_s}^{m-2 }}\simeq \pi_s^+\paren{  \overbrace{F_s\ldots F_t}^{m-3 }F_s^2\overbrace{F_t\ldots F_s}^{m -3}}\simeq \ldots \simeq \pi_s^+\paren{R}=R  
   \end{equation} 
   Note that it is crucial that the start and ends of the expressions above end in $s$ as  $\pi_s^+(F_s)\simeq 0$ by \cref{pis0}$(ii)$ but $\pi_s^+(F_t)\not\simeq 0$! When $m$ is even, \cref{fteq} has $F_s^2$ in the middle (for the case $m=2$, $\ft_{\set{s,t}/t}=F_s^2$) but the start and end still consists of $F_s$ so the same arguments above also show \cref{pifteq} in this case. \\

   Using \cref{semicor}, $X^\bullet$ fits into the distinguished triangle
   \[  \pi_s^-(X^\bullet)\xrightarrow{\eta} X^\bullet \to C^-(X^\bullet)\xrightarrow{+1}  \]
   where $C^-(X^\bullet)=\cone(\eta)\in \Uc^-_t$. Tensoring with $ \ft_{\set{s,t}/t}$ on the left gives us the distinguished triangle
   \[  \ft_{\set{s,t}/t}\otimes_R \pi_s^-(X^\bullet)\to  \ft_{\set{s,t}/t}\otimes_R X^\bullet \to  \ft_{\set{s,t}/t}\otimes_R C^-(X^\bullet)\xrightarrow{+1}  \]
  $\ft_{\set{s,t}/t}\otimes_R C^-(X^\bullet)\in \Uc_t^+$ by \cref{tenequiv} so the last term is homotopy equivalent to 0 after applying the triangulated functor $\pi_s^+$. Thus by bilinearity and \cref{pifteq} we see that
  \begin{equation}
  \label{relsereq}
  \pi_s^-(X^\bullet)\simeq \pi_s^+\paren{\ft_{\set{s,t}/t}}\otimes_R \pi_s^-(X^\bullet)= \pi_s^+\paren{\ft_{\set{s,t}/t}\otimes_R \pi_s^-(X^\bullet)}=   \pi_s^+\paren{\ft_{\set{s,t}/t}\otimes_R X^\bullet} 
  \end{equation}
\end{proof}

\subsection{Dualities}

\begin{lemma}[{\cite[Exercise 18.13]{EMTW20}, \cite[Proposition 5.9]{Soe07}}]
\label{soelem}
    Let $\hf$ be a realization of a Coxeter group $W$ and let $B, B^\prime\in \sbim(\hf, W)$. Then 
    \[ \Db(B\otimes_R B^\prime) =\Db(B)\otimes_R \Db(B^\prime)  \]
\end{lemma}

Let us now introduce two other ``duality" functors on $\sbim(\hf, W)$ and their interactions with $\Db$.

\begin{definition}
\begin{enumerate}[(1)]
    \item $\omega : \sbim(\hf, W) \to \sbim(\hf, W)$
is the \un{covariant} graded, anti-monoidal functor which sends $B_s \mapsto B_s$, fixes morphisms, and satisfies $\omega(MN) = \omega(N)\omega(M)$ and $\omega(M(1)) = \omega(M)(1)$.
\item $(-)^\vee : \sbim(\hf, W) \to \sbim(\hf, W)^{op}$ is the \un{contravariant} anti-graded, anti-monoidal functor which sends $B_s \mapsto B_s$, flips morphisms diagrammatically, and satisfies 
$(MN)^\vee = N^\vee M^\vee $ and $M(1)^\vee = M^\vee(-1).$
\end{enumerate}
\end{definition}

Because $\Db(B_s)\cong B_s$ \cite[Exercise 18.12]{EMTW20}, it follows from \cref{soelem} that

\begin{equation}
\label{wdveq}
    \omega (\Db(-))=(-)^\vee =\Db(\omega(-))
\end{equation}

The monoidal or anti-monoidality of these 3 functors extend to $K^b(\sbim(\hf, W))$ as well as \cref{wdveq}. They are also all fully faithful and because $\hh^k(MN)\xcong{v.s.} \hh^k(NM)$ it follows that $\hh^k\circ \omega\xcong{v.s.}\hh^k$ and thus 
\[ \hh^k(\Db(X^\bullet))\xcong{v.s.} \hh^k\paren{ \paren{X^\bullet}^\vee} \]
but not as $R-$modules! Because $B_s$ is a Frobenius algebra object in $\sbim(\hf, W)$ (see \cite[Section 8.2]{EMTW20}), the functor $B_s\otimes_R(-)$ is self adjoint. Likewise with $(-)\otimes_R B_s$ from which it follows that for all $B\in \sbim(\hf, W)$
\[ \hom_{R^e}(B \otimes_R M, N)\cong \hom_{R^e}( M, B^\vee \otimes_R N) \qquad \text{and}\qquad \hom_{R^e}( M\otimes_R B, N)\cong \hom_{R^e}( M, N\otimes_R B^\vee ) \]

As a result, for any $X^\bullet\in K^b(\sbim(\hf, W))$ 
\[ \uhom_{R^e}(X^\bullet, Y^\bullet)\xcong{\text{right }R} \uhom_{R^e}(R, (X^\bullet)^\vee \otimes_R Y^\bullet)\xcong{\text{right }R} \uhom_{R^e}(  (Y^\bullet)^\vee \otimes_R X^\bullet, R ) \]

\subsubsection{Final Computation}

\begin{lemma}
    Given $B\in \sbim(\hf, I_2(m))$, $\hh^0(B)=\pi^-_t \circ \pi^-_s (B)$ and  $\hh^0(B)=\pi^+_t \circ \pi^+_s (B)$.
\end{lemma}
\begin{proof}
    As vector spaces we have that $V^*=\kb \rho_s \oplus \kb \rho_t \oplus (V^*)^{s,t}$. Let $\set{e_1^*, \ldots, e_r^*}$ be a basis for $(V^*)^{s,t}$. Then using the center and co-center interpretation of $\hh^0$ and $\hh_0$, respectively, we see that 
    \[  \hh^0(B)= \pi^-_t \circ \pi^-_s \circ \pi^-_{e_1^*}\ldots \circ \pi^-_{e_r^*}, \qquad \hh_0(B)= \pi^+_t \circ \pi^+_s \circ \pi^+_{e_1^*}\ldots \circ \pi^+_{e_r^*}  \]
    However, $\pi^-_{e_i^*}(B)\cong  \pi^+_{e_i^*}(B)\cong B$ for all $B\in\sbim(\hf, I_2(m)) $ and $1\le i \le r$ and thus the lemma follows. 
\end{proof}

\begin{lemma}
  $\pi^-_t, \pi^{+}_t $ are the right and left adjoints to the inclusion functor $\abrac{\boxed{R}}\to K^b(\sbim(\hf, \abrac{t}))$ and 
  \begin{equation}
  \label{rank1releq}
       \pi_t^-(Z^\bullet)\simeq    \pi_s^+\paren{F_t^2\otimes_R Z^\bullet} \simeq \pi_s^+\paren{Z^\bullet\otimes_R F_t^2}  \qquad \forall Z^\bullet \in K^b(\sbim(\hf, \abrac{t}))
  \end{equation}
\end{lemma}
\begin{proof}
   The first statement follows from \cref{partialprop} while \cref{rank1releq} follows from the same proof as in \cref{relsertheorem} where we instead have $\ft_{\set{s,t}/t}\longleftrightarrow F_t^2$, $\Uc^-_t\longleftrightarrow \text{span of } \set{F_t^{-1}}$, and $\Uc^+_t\longleftrightarrow \text{span of}  \set{F_t}$.
\end{proof}

\begin{remark}
    $t$ and $s$ can be switched in the above two lemmas.
\end{remark}

\begin{theorem}[Serre Duality]
    For any two complexes $X^\bullet, Y^\bullet \in K^b(\sbim(\hf, I_2(m)))$ one has the following homotopy equivalences of complexes of right (or left) $R-$modules.
\[
\uhom_{R^e}(X^\bullet,Y^\bullet) \simeq \uhom_{R^e}(Y^\bullet, \mathrm{FT}_{\set{s,t}}^{-1} \otimes_R X^\bullet)^\vee \simeq \uhom_{R^e}(\mathrm{FT}_{\set{s,t}} \otimes_R Y^\bullet, X^\bullet)^\vee 
\]
\end{theorem}
\begin{proof}
    Let $B\in \sbim(\hf, I_2(m))$. From the universal coefficient spectral sequence we have that
    \[ \ext_R^p(\hh_q(B), R)\implies \ext_{R^e}^{p+q}(R, \hom_R(B, R)) \]
    But by \cref{li36}, $\hh^q(B)$ and thus $\hh_q(B)$ (see \cite[Lemma 3.4]{GHMN19}) is a free $R$ module and thus
    \begin{equation}
    \label{dhheq}
      \hh_q(B)^\vee=  \Db(\hh_q(B))\cong \hh^q(\Db(B)) \qquad \qquad \forall B\in \sbim(\hf, I_2(m))
    \end{equation}
    where the first equality is because $\omega=\id$ on complexes of $R-$modules. Now we compute
    \begin{align*}
        \uhom_{R^e}(X^\bullet, Y^\bullet)&\cong  \uhom_{R^e}(R, (X^\bullet)^\vee \otimes_R Y^\bullet)=\hh^0((X^\bullet)^\vee \otimes_R Y^\bullet) =\pi_t^- \circ \pi_s^- (  (X^\bullet)^\vee \otimes_R Y^\bullet ) \\
        & \xsim{\cref{relsereq}}  \pi_t^- \pi_s^+\paren{\ft_{\set{s,t}/t}\otimes_R (X^\bullet)^\vee \otimes_R Y^\bullet} \xsim{\cref{rank1releq}} \pi_t^+\paren{ F_t^2\pi_s^+\paren{\ft_{\set{s,t}/t}\otimes_R (X^\bullet)^\vee \otimes_R Y^\bullet} }
    \end{align*}
    By \cref{lembilinear}, we can pull $F_t^2$ inside $\pi_s^+$ and thus the RHS equals 
    \begin{align*}
        \hh_0\paren{\ft_{\set{s,t}}\otimes_R (X^\bullet)^\vee \otimes_R Y^\bullet} &\xcong{\cref{dhheq}} \hh^0\paren{ \Db\paren{ \ft_{\set{s,t}}\otimes_R (X^\bullet)^\vee \otimes_R Y^\bullet } } ^\vee \xcong{} \uhom_{R^e}\paren{R, \Db\paren{ \ft_{\set{s,t}})\otimes_R \Db((X^\bullet)^\vee) \otimes_R \Db( Y^\bullet) }}^\vee \\
        &\cong \uhom_{R^e}\paren{\Db( Y^\bullet)^\vee, \Db\paren{ \ft_{\set{s,t}})\otimes_R \Db((X^\bullet)^\vee)  }}^\vee\xcong{\cref{wdveq}} \uhom_{R^e}\paren{\omega( Y^\bullet), \omega\paren{ \ft_{\set{s,t}}^{-1}}\otimes_R \omega(X^\bullet)  }^\vee \\
        &\cong\uhom_{R^e}\paren{\omega( Y^\bullet), \omega\paren{ \ft_{\set{s,t}/t}^{-1}\otimes_R \omega(X^\bullet)}  }\cong \uhom_{R^e}\paren{ Y^\bullet,  \ft_{\set{s,t}/t}^{-1}\otimes_R \omega(X^\bullet)  }^\vee
    \end{align*}
\end{proof}

\section{Remarks about the General Case}
\label{gensect}

Given a general Coxeter system $(W, S)$ and a parabolic subgroup $W_I$, where $I\subset S$, it was proven in \cite[Appendix A]{GHMN19} that the left and right adjoints to the natural inclusion functor $K^b(\sbim(\hf, W_I))\to K^b(\sbim(\hf, W))$ exists. We conjecture that they have an explicit form, namely

\begin{conjecture}
\label{conj2}
    Let $\hf$ be a realization of $W$ with fundamental weights $\set{\rho_s}_{s\in S}$. Given a subset $I\subset S$, let $\set{s_1, \ldots, s_\ell}$ denote the complement of $I$ in $S$. Let $\pi_L, \pi_R:K^b(\sbim(\hf, W))\to  K^b(\sbim(\hf, W_I))$ be the left and right adjoints to the inclusion functor $K^b(\sbim(\hf, W_I))\to K^b(\sbim(\hf, W))$. Then 
    \[ \pi_L\cong \pi^+_{s_1}\circ \ldots \circ \pi^+_{s_\ell} , \qquad \qquad  \pi_R\cong \pi^-_{s_1}\circ \ldots \circ \pi^-_{s_\ell}  \]
   This is independent of the order chosen for the complement of $I$ in $S$.
\end{conjecture}

\begin{remark}
   The functors used in the proof of \cref{ghmnconj} in Type A for the special case of subgroups $S_r\times (S_1)^{n-r}\subset S_n$ do not fit the framework above because $\set{x_2, \ldots, x_n}$ are not fundamental weights for the $\Cb^n$ realization of $S_n$. However, $x_n$ is a fundamental weight. Thus, for the corresponding subgroup $S_r\times (S_1)^{n-r}$, the composition $$\pi^{\pm}=\pi^{\pm}_{x_{r+1}}\circ \ldots \circ \pi^{\pm}_{x_n}$$ 
   in this order gives the correct adjoints, but using these functors for the other parabolic subgroups of $S_n$ will fail.
\end{remark}

One can show that it suffices to consider the maximal parabolic case $W\setminus\set{s}\subset W$ for \cref{conj2}. Similar to \cref{adjcor}, \cref{conj2} follows from showing that $\pi^{\pm}_s$ sends Soergel Bimodules to Soergel Bimodules, namely
\begin{equation}
\label{pisbimeq}
    \pi^{\pm}_s(B_w)\in \sbim(\hf, W_{S\setminus \set{s}}) \quad \forall w\in W
\end{equation}

If one works with the homotopy category, it also suffices to show $\pi^{\pm}_s$ sends Bott-Samelsons to Soergel Bimodules. For the maximal parabolic $S_{n-1}\times S_1\subset S_n$ considered in \cite{GHMN19}, \cref{pisbimeq} easily follows from the fact that any $w\in S_n \setminus (S_{n-1}\times S_1)$ can be written as $w=us_{n-1} v$ where 
$u, v\in S_{n-1}\times S_1$. Another easy consequence of this fact is
\begin{equation}
\label{kereq}
    \ker \pi^+_{s_{n-1}}=\abrac{F_w| w\in W, \, w\not\in S_{n-1}\times S_1} \qquad  \ker \pi^-_{s_{n-1}}=\abrac{F_w^{-1} |w\in W, \, w\not\in S_{n-1}\times S_1}
\end{equation}
For other maximal parabolics\footnote{even in Type $A$.} of a Coxeter group, this fact is not true as the double coset space has size greater than 2. We overcome this obstacle in this paper by employing the tools developed in \cite{Li25}, arising from the diagrammatic presentation of Ext-enhanced Soergel bimodules in rank 2. We expect that a diagrammatic presentation of Ext-enhanced Soergel bimodules for arbitrary Coxeter groups $W$ will yield \cref{pisbimeq} and the analogous \cref{kereq} for $W_I\subset W$, from which \cref{conj2} and \cref{ghmnconj} will readily follow.   

\section{Computations of Triply-graded link homology}\label{whiteheadsect}
\begin{prop}[Gaussian Elimination {\cite[Proposition 5.16]{Eli18}}]
\label{gausselim}
Let $\Ac$ be an additive category. Suppose that $C^\bullet\in C(\Ac)$ has the form 
\begin{equation}
\label{gausseq}
    \ldots\to C^{k-1} \xrightarrow{\begin{pmatrix}
b\\
a
\end{pmatrix}} X\oplus D^k  \xrightarrow{\begin{pmatrix}
\psi & d \\
e & c
\end{pmatrix}} Y\oplus  D^{k+1} \xrightarrow{\begin{pmatrix}
g & f\\
\end{pmatrix}} C^{k+2} 
\end{equation} 
where $\psi: X\to Y$ is an isomorphism. Then $C^\bullet\simeq \widetilde{C^\bullet}$ in $K(\Ac)$ where $\widetilde{C^\bullet}$ is the complex
\[  \scalemath{1.2}{\ldots \to C^{k-1}\xrightarrow{a} D^k \xrightarrow{c-e\psi^{-1}d}D^{k+1} \xrightarrow{f} C^{k+2} \to \ldots } \]
with projection 
\begin{tikzcd}
     X\oplus D^k   \ar[d, "(0 \ 1 )"]\ar[r] & Y\oplus D^{k+1}   \ar[d, "(-e\psi^{-1} \ 1)"]  \\
     D^k  \ar[r] & D^{k+1}   
\end{tikzcd}
and inclusion 
\begin{tikzcd}
      D^k  \ar[r]\ar[d, "\begin{bmatrix}
    -\psi^{-1} d \\
    1
    \end{bmatrix}
    "] & D^{k+1}  \ar[d, "\begin{bmatrix}
    0 \\
    1
    \end{bmatrix}
    "]  \\
    X\oplus D^k  \ar[r] & Y\oplus D^{k+1}
\end{tikzcd}.
\end{prop}

\begin{corollary}
In the above proposition, if $d$ or $e=0$ (for example if $D^k=0$ or $D^{k+1}=0$) we can remove $X$ and $Y$ without worrying about any changes to the differential.
\end{corollary}

\subsection{The Whitehead link}

The Whitehead link $L5a1$ has a braid representative given by $s^{-2}ts^{-1}t$. Let $W^\bullet=F_s^{-2}F_tF_s^{-1}F_t$. As complexes, we have that (note $R=\kb[\alpha_s, \alpha_t]$ below)
\begin{equation*}
F_s^{-2}\simeq R(-2) \xrightarrow{\runit} B_s(-1) \xrightarrow{\ds} \boxed{B_s(1)}  \qquad F_tF_s^{-1}=\scalemath{0.92}{
    B_t(-1) \xrightarrow{ \begin{pmatrix}
-\bcounit \\
\bzero \runit
\end{pmatrix} } \boxed{R\oplus B_t B_s} \xrightarrow{\begin{pmatrix}
\raisebox{-5pt}{ \runit}  &  \bcounit \rzero \\     
\end{pmatrix}} B_s(1) } 
\end{equation*}

Using \cref{gausselim}, the complex $F_tF_s^{-1}B_t$ is homotopy equivalent to $B_t(-2)\xrightarrow{\begin{tikzpicture}[scale=-0.6,line cap=round,line join=round]
  \draw[very thick,blue]
    (-0.5,0) .. controls (-0.5,0.7) and (-0.3,0.9) .. (0,0.9)
             .. controls ( 0.3,0.9) and ( 0.5,0.7) .. (0.5,0);
  \draw[very thick,blue] (0,0.9) -- (0,1.4);
  \draw[very thick,red,] (0,0) -- (0,0.5);
  \node[dot, red] at (0,0.5) {};
  \node[btridot] at (0.5,0.4) {$\scalemath{0.85}{d_t}$};
  \node at (1,0.4) {$-$};
\end{tikzpicture}
}B_tB_sB_t \xrightarrow{ \bcounit \rzero \bzero }B_sB_t(1)$ where

\[ \begin{pmatrix}
\psi & d \\
e & c
\end{pmatrix}=\begin{pmatrix}
-\bzero & -\boxed{\alpha_t/2} \bzero \\
\begin{tikzpicture}[scale=-0.6,line cap=round,line join=round]
  \draw[very thick,blue]
    (-0.5,0) .. controls (-0.5,0.7) and (-0.3,0.9) .. (0,0.9)
             .. controls ( 0.3,0.9) and ( 0.5,0.7) .. (0.5,0);
  \draw[very thick,blue] (0,0.9) -- (0,1.4);
  \draw[very thick,red,] (0,0) -- (0,0.5);
  \node[dot, red] at (0,0.5) {};
\end{tikzpicture} & \begin{tikzpicture}[scale=-0.6,line cap=round,line join=round]
  \draw[very thick,blue]
    (-0.5,0) .. controls (-0.5,0.7) and (-0.3,0.9) .. (0,0.9)
             .. controls ( 0.3,0.9) and ( 0.8,0.9) .. (1,0);
  \draw[very thick,blue] (0,0.9) -- (0,1.4);
  \draw[very thick,red,] (0,0) -- (0,0.3);
  \node[dot, red] at (0,0.3) {};
  \node at (0.5,0.3) {$\scalemath{0.7}{\frac{\alpha_s}{2}}$};
\end{tikzpicture}
\end{pmatrix} \]
Using the inclusion map into the original complex $F_tF_s^{-1}B_t$, we see that $F_tF_s^{-1}F_t$ is homotopy equivalent to

\[\xymatrix@R=4em{
    B_t \ar[r]^-{ \mathrel{ \begin{pmatrix}
-\bcounit \\
\bzero \runit
\end{pmatrix}}} & R(1)\oplus B_tB_s(1) \ar[r]^-{ \mathrel{ \begin{pmatrix}
\runit & \bcounit \rzero 
\end{pmatrix}}}  & B_s(2) \\
B_t(-2) \ar[u]^-{
    \dt
} \ar[r]^-{\begin{tikzpicture}[scale=-0.6,line cap=round,line join=round]
  \draw[very thick,blue]
    (-0.5,0) .. controls (-0.5,0.7) and (-0.3,0.9) .. (0,0.9)
             .. controls ( 0.3,0.9) and ( 0.5,0.7) .. (0.5,0);
  \draw[very thick,blue] (0,0.9) -- (0,1.4);
  \draw[very thick,red,] (0,0) -- (0,0.5);
  \node[dot, red] at (0,0.5) {};
  \node[btridot] at (0.5,0.4) {$\scalemath{0.85}{d_t}$};
  \node at (1,0.4) {$-$};
\end{tikzpicture} } & \boxed{B_tB_sB_t} \ar[u]_{\begin{pmatrix}
    0 \\
    \bzero \rzero \bcounit
\end{pmatrix}} \ar[r]^{\bcounit \rzero \bzero} & B_s B_t(1) \ar[u]_{-\rzero \bcounit}
} \]
Applying $\pi_t^-(-)$ produces the complex on the left below while applying $\pi_t^+(-)(2)$ produces the complex on the right

\begin{equation}
\label{twocomplex}
 \xymatrix@R=4em@C=3.5em{
R(-1) \ar[r]^{\begin{pmatrix}
    -\alpha_t \\
    \runit
\end{pmatrix}}& R(1)\oplus B_s \ar[r]^{\begin{pmatrix}
    \runit & \alpha_t \rzero
\end{pmatrix}}& B_s(2) \\
R(-3) \ar[u]^{0} \ar[r]^-{\begin{pmatrix}
    \runit \\
    -\alpha_t 
\end{pmatrix}} & \boxed{B_s(-2)\oplus R(-1)} \ar[u]_{\begin{pmatrix}
    0 & 0 \\
     \rzero \alpha_t & \runit
\end{pmatrix}} \ar[r]^-{ \quad \begin{pmatrix}
      \alpha_t \rzero  & \runit
\end{pmatrix}} & B_s \ar[u]_{- \rzero \alpha_t}
} \qquad 
\xymatrix@R=4em@C=3.5em{
R(3) \ar[r]^{\begin{pmatrix}
    -1 \\
    \runitd
\end{pmatrix}}& R(3)\oplus B_s(4) \ar[r]^{\begin{pmatrix}
    \runitd & 1
\end{pmatrix}}& B_s(4) \\
R(1) \ar[u]^{0} \ar[r]^-{\begin{pmatrix}
    0 \\
    \alpha_s+\alpha_t
\end{pmatrix}} & \boxed{B_s(4)\oplus R(3)} \ar[u]_{\begin{pmatrix}
    0 & 0 \\
    1 & 0
\end{pmatrix}} \ar[r]^-{ \quad \begin{pmatrix}
    1 & 0
\end{pmatrix}} & B_s(4) \ar[u]_{-1}
}
\end{equation}

where $\ell_{\alpha_t}, r_{\alpha_t}$ are left and right multiplication of $\alpha_t$. Using the LHS, $\pi_t^-(W^\bullet)=F_s^{-2} \pi_t^-(F_tF_s^{-1}F_t)$ is homotopic to

\[
\xymatrix@R=6em@C=5.5em{
B_s(-2) \ar[r]^-{\begin{pmatrix}
    0 \\
   \rzero \runit \\
    -\rzero \alpha_t 
\end{pmatrix}} & \boxed{B_s\oplus (B_sB_s(-1) \oplus B_s ) } \ar[r]^-{\begin{pmatrix}
    -\rzero\alpha_t & 0 & 0 \\
   \rzero \runit & \rzero \rzero \alpha_t & \rzero \runit \\
    0 & \rzero \alpha_t \rzero & \rzero \runit 
\end{pmatrix}}&  (B_s(2) \oplus B_sB_s(1))\oplus B_sB_s(1) \ar[r]^-{\begin{pmatrix}
   \rzero \runit & \rzero \alpha_t \rzero & -\rzero \rzero \alpha_t 
\end{pmatrix}}& B_sB_s(3) \\  
B_s(-4) \ar[u]_{-\ds} \ar[r]^-{d_{-1}} & B_s(-2)\oplus (B_sB_s(-3) \oplus B_s(-2) )\ar[u]_{\scalemath{0.9}{\begin{pmatrix}
\ds & 0 & 0 \\
0 & \ds \rzero & 0 \\
0 & 0 & \ds 
\end{pmatrix}}}  \ar[r]^-{d_0}&  (B_s \oplus B_sB_s(-1))\oplus B_sB_s(-1) \ar[r]^-{d_1} \ar[u]_{\scalemath{0.9}{\begin{pmatrix}
-\ds & 0 & 0 \\
0 & -\ds \rzero & 0 \\
0 & 0 & -\ds \rzero
\end{pmatrix}}} & B_sB_s(1) \ar[u]_{\ds \rzero} \\ 
R(-5) \ar[u]_{-\runitd } \ar[r]^-{\overline{d_{-1}}}& R(-3)\oplus (B_s(-4)\oplus R(-3)) \ar[u]_{\scalemath{0.9}{\begin{pmatrix}
\runitd & 0 & 0 \\
0 & \runit \rzero & 0 \\
0 & 0 & \runitd
\end{pmatrix}}}  \ar[r]^-{\overline{d_0}}& (R(-1)\oplus B_s(-2))\oplus B_s(-2)\ar[u]_{\scalemath{0.9}{\begin{pmatrix}
-\runitd & 0 & 0 \\
0 & -\runit \rzero & 0 \\
0 & 0 & -\runitd \rzero
\end{pmatrix}}}  \ar[r]^-{\overline{d_1}} & B_s \ar[u]_{\runit\rzero}
}
\]
where $d_i$ is the same as the differential above it and $\overline{d_i}$ is $d_i$ without \rzero \ on the left. To calculate $\overline{\hhh}^{A=0}$ we apply $\hh^0$ and obtain the complex (we are ordering the summands of the diagonals from left to right and the numbers indicate which summand of the diagonal the basis vectors are from)
\[ \scalemath{0.88}{ \sbrac{\boxed{1}(-3)\oplus \runitd (-4)\oplus \boxed{1}(-3) }_2 \xrightarrow{\begin{pmatrix}
     0 & 0 &0 \\
     1 & 0 &0 \\
     0 & 1 &0 \\
     0 & 0 &0 \\
     0 & 0 &1 \\
     -\alpha_t & 0 &0 \\
    1 & \alpha_t &1 \\
     0 & \alpha_t &1 \\
\end{pmatrix}} \sbrac{\runitd(-2)}_1\oplus \sbrac{\runitd(-2)\oplus \paren{\runitd\runitd(-3)\oplus \rcup (-3) }\oplus \runitd(-2) }_2\oplus \sbrac{\boxed{1}(-1)\oplus \runitd(-2)\oplus \runitd(-2)  }_3} \]
\[\scalemath{0.9}{ \xrightarrow{\begin{pmatrix}
    0 & 0 & 0 & 0 & 0 & 0 & 0 &0 \\
    1 & 0 & 0 & -1 & 0 & 0 & 0 &0 \\
    0 & 0 & 0 &\alpha_s & 0 & 0 &0 & 0 \\
    -\alpha_t & 0 & 0 & 0 & 0 & 0 &0 & 0 \\
    0 & -\alpha_t & 0 & 0 & 0 & -1 & 0 & 0 \\
    0 & 1 & \alpha_t & 0 & 1 & 0 & -1 & 0 \\
    0 & 0 & 0 & \alpha_t & 0 & 0 & 0 & 0 \\
    0 & 0 & \alpha_t & -1 & 1 & 0 & 0 & -1 \\
    0 & 0 & 0 & \alpha_t+\alpha_s & 0 & 0 & 0 & 0 \\
    0 & 0 & 0 & 0 & 0 & 1 & \alpha_t & -\alpha_t \\
\end{pmatrix}}\boxed{\sbrac{\runitd \oplus\paren{ \runitd\runitd(-1) \oplus \rcup (-1) }\oplus \runitd  }_1\oplus \sbrac{ \runitd \oplus \paren{\runitd \runitd (-1)\oplus \rcup (-1)}\oplus \paren{\runitd \runitd (-1) \oplus \rcup (-1)} }_2\oplus \sbrac{\runitd}_3}} \]
\[ \xrightarrow{\begin{pmatrix}
    -\alpha_t & 0 & 0 & 0 & 0 & 0&0 &0 &0 &0 \\
    1 & \alpha_t & 0 & 1 & 0 & 0 & 1 & 0 & 0 & 0 \\
    0 & 0 & \alpha_t & 0 & 0 & 0 & -\alpha_s & 0 & 0 & 0 \\
    0 & \alpha_t & -1 & 1 & 0 & 0 & 0 & 0 & 1 & 0 \\
    0 & 0 & \alpha_t +\alpha_s & 0 & 0 &0 & 0 &0 & -\alpha_s & 0\\
    0 & 0 & 0 & 0 & 1 & \alpha_t & -1 & -\alpha_t & 0 & 1\\
    0 & 0 & 0 & 0& 0  & 0 & \alpha_t+\alpha_s & 0 & -\alpha_t & 0
\end{pmatrix}} \sbrac{\runitd(2)\oplus \paren{\runitd \runitd (1)\oplus \rcup (1)} \oplus \paren{\runitd\runitd (1)\oplus \rcup (1)}}_1 \oplus \sbrac{\runitd \runitd(1)\oplus \rcup(1) }_2  \]
\[ \xrightarrow{\begin{pmatrix}
    1 & \alpha_t & -1 & -\alpha_t & 0  & 0 &-1 \\
    0 & 0 &\alpha_t+\alpha_s & 0 & -\alpha_t  & 0 & \alpha_s
\end{pmatrix}} \runitd \runitd(3) \oplus \rcup(3)\]
Note that $\hh^0(\runitd)=1$, and so we have applied Gaussian elimination to $\hh^0(R(-5)\xrightarrow{-\runitd} B_s(-4))$ at the cohomological degree $-2$ spot. The nonzero cohomology at $A=0$ will be
\[ \paren{\overline{\hhh}^{A=0}}^{T=1}=\begin{bmatrix}
    0 & 0 & 1 & 0 & 1 & 0 & -1
\end{bmatrix} ^T \kb (1)=\kb(1), \qquad \paren{\overline{\hhh}^{A=0}}^{T=2}= \paren{\rcup} \kb(3)=\kb(3)  \]
For $A=1$, the complex $\pi_s^+(\pi_t^-(W^\bullet))$ will be
\[ \scalemath{0.8}{ \dboxed{\alpha_s^\vee}\xrightarrow{\begin{pmatrix}
    -\alpha_s \\
    0 \\
    \alpha_s \\
    -\alpha_t
\end{pmatrix}}\sbrac{\rhunit(-2)}_1\oplus\sbrac{\dboxed{\alpha_s^\vee}(-3)\oplus \rhunit (-4)\oplus \dboxed{\alpha_s^\vee}(-3) }_2 \xrightarrow{\begin{pmatrix}
    0&  0 & 0 &0 \\
    0& \alpha_s & 0 &0 \\
     1& 0 & 1 &0 \\
     0&0 & 0 &0 \\
     -\alpha_t & 0 & 0 &\alpha_s \\
     0& -\alpha_t & 0 &0 \\
    0& \alpha_s & \alpha_t &\alpha_s \\
  0&   0 & \alpha_t &\alpha_s \\
\end{pmatrix}} \sbrac{\rhunit(-2)}_1\oplus \sbrac{\rhunit(-2)\oplus \paren{\rhunit\runitd(-3)\oplus \rhcup (-3) }\oplus \rhunit(-2) }_2\oplus \sbrac{\boxed{\alpha_s^\vee}(-1)\oplus \rhunit(-2)\oplus \rhunit(-2)  }_3} \]

\[\scalemath{0.9}{ \xrightarrow{\begin{pmatrix}
    0 & 0 & 0 & 0 & 0 & 0 & 0 &0 \\
    1 & 0 & 0 & -1 & 0 & 0 & 0 &0 \\
    0 & 0 & 0 &\alpha_s & 0 & 0 &0 & 0 \\
    -\alpha_t & 0 & 0 & 0 & 0 & 0 &0 & 0 \\
    0 & -\alpha_t & 0 & 0 & 0 & -\alpha_s & 0 & 0 \\
    0 & 1 & \alpha_t & 0 & 1 & 0 & -1 & 0 \\
    0 & 0 & 0 & \alpha_t & 0 & 0 & 0 & 0 \\
    0 & 0 & \alpha_t & -1 & 1 & 0 & 0 & -1 \\
    0 & 0 & 0 & \alpha_t+\alpha_s & 0 & 0 & 0 & 0 \\
    0 & 0 & 0 & 0 & 0 & \alpha_s & \alpha_t & -\alpha_t \\
\end{pmatrix}}\boxed{\sbrac{\rhunit \oplus\paren{ \rhunit\runitd(-1) \oplus \rhcup (-1) }\oplus \rhunit  }_1\oplus \sbrac{ \rhunit \oplus \paren{\rhunit \runitd (-1)\oplus \rhcup (-1)}\oplus \paren{\rhunit \runitd (-1) \oplus \rhcup (-1)} }_2\oplus \sbrac{\rhunit}_3}} \]
\[ \xrightarrow{\begin{pmatrix}
    -\alpha_t & 0 & 0 & 0 & 0 & 0&0 &0 &0 &0 \\
    1 & \alpha_t & 0 & 1 & 0 & 0 & 1 & 0 & 0 & 0 \\
    0 & 0 & \alpha_t & 0 & 0 & 0 & -\alpha_s & 0 & 0 & 0 \\
    0 & \alpha_t & -1 & 1 & 0 & 0 & 0 & 0 & 1 & 0 \\
    0 & 0 & \alpha_t +\alpha_s & 0 & 0 &0 & 0 &0 & -\alpha_s & 0\\
    0 & 0 & 0 & 0 & 1 & \alpha_t & -1 & -\alpha_t & 0 & 1\\
    0 & 0 & 0 & 0& 0  & 0 & \alpha_t+\alpha_s & 0 & -\alpha_t & 0
\end{pmatrix}} \sbrac{\rhunit(2)\oplus \paren{\rhunit \runitd (1)\oplus \rhcup (1)} \oplus \paren{\rhunit\runitd (1)\oplus \rhcup (1)}}_1 \oplus \sbrac{\rhunit \runitd(1)\oplus \rhcup(1) }_2  \]
\[ \xrightarrow{\begin{pmatrix}
    1 & \alpha_t & -1 & -\alpha_t & 0  & 0 &-1 \\
    0 & 0 &\alpha_t+\alpha_s & 0 & -\alpha_t  & 0 & \alpha_s
\end{pmatrix}} \rhunit\runitd(3) \oplus \rhcup(3)\]
The nonzero cohomology of $\pi_s^+(\pi_t^-(W^\bullet))$ will then be 
\[\scalemath{0.9}{\h^{T=-1}(\pi_s^+(\pi_t^-(W^\bullet))=\begin{bmatrix}
    0 & 0 & 0 & 0 & 1 & 0 & 1 & 1
\end{bmatrix}^T\kb(-2)=\kb (1), \qquad \h^{T=0}(\pi_s^+(\pi_t^-(W^\bullet))=\begin{bmatrix}
    0 & 0 & 0 & 0 & 1 & 0 & 0 & 0 & -1
\end{bmatrix}^T \kb=\kb(3)} \]
\[\h^{T=1}(\pi_s^+(\pi_t^-(W^\bullet))=\begin{bmatrix}
    0 & 0 & 1 & 0 & 1 & 0 & -1
\end{bmatrix}^T \kb (1)=\kb(5), \qquad \h^{T=2}(\pi_s^+(\pi_t^-(W^\bullet))= \paren{\rhcup} \kb(3)=\kb(7) \]
Now, because the RHS of \cref{twocomplex} has many contractible summands, we can remove them to yield the complex $R(1)\xrightarrow{\alpha_s+\alpha_t}\boxed{R(3)}$. Since $\pi_s^-(F_s^{-2})\simeq \boxed{R}$, it follows that
$$\pi_s^-(\pi_t^+(W^\bullet))\simeq R(1)\xrightarrow{\alpha_s+\alpha_t}\boxed{R(3)}$$ 
The nonzero cohomology of $\pi_s^+(\pi_t^-(W^\bullet))$ will then be
\[\h^{T=0}(\pi_s^-(\pi_t^+(W^\bullet)))= \frac{R(3)}{(\alpha_s+\alpha_t)} \]
It follows that the nonzero cohomology at $A=1$ will be
\[ \paren{\overline{\hhh}^{A=1}}^{T=-1}\cong\kb(1), \qquad \paren{\overline{\hhh}^{A=1}}^{T=0}\cong\kb(3)\oplus \frac{R(3)}{(\alpha_s+\alpha_t)}, \quad  \paren{\overline{\hhh}^{A=1}}^{T=1}\cong\kb(5), \quad  \paren{\overline{\hhh}^{A=1}}^{T=2}\cong\kb(7)  \]

Finally for $A=2$, the complex $\hh^1 \pi_t^+(W^\bullet)$ is homotopy equivalent to 
\[\dboxed{\alpha_s^\vee}(-1)\xrightarrow{\begin{pmatrix}
    \alpha_s+\alpha_t \\
    \alpha_s
\end{pmatrix}} \dboxed{\alpha_s^\vee}(1)\oplus \rhunit \xrightarrow{\begin{pmatrix}
    \alpha_s & -(\alpha_s+\alpha_t) \\
    0 & 0
\end{pmatrix}} \rhunit(2)\oplus \rhunit(2)\xrightarrow{\begin{pmatrix}
    0 & \alpha_s+\alpha_t
\end{pmatrix}}\boxed{\rhunit (4)} \]
and thus the nonzero cohomology at $A=2$ is given by
\[  \paren{\overline{\hhh}^{A=1}}^{T=-1}\cong \rhunit \kb(2)=\kb(5), \quad  \paren{\overline{\hhh}^{A=1}}^{T=0}\cong \rhunit R/(\alpha_s+\alpha_t)(4)= R/(\alpha_s+\alpha_t)(7)   \]

The total Poincare series for $\overline{\hhh}(F_s^{-2}F_tF_s^{-1}F_t)$ will then be
\begin{align*}
    \overline{\mathscr{P}(A, Q, T)}(F_s^{-2}F_tF_s^{-1}F_t)&=TQ^{-1}+T^2Q^{-3}+A\paren{T^{-1}Q^{-1} + Q^{-3}+\frac{Q^{-3}}{(1-Q^2)} +TQ^{-5}+T^2Q^{-7} } +A^2 \paren{ T^{-1}Q^{-5}+\frac{Q^{-7}}{(1-Q^2)}} 
\end{align*}

Substituting $A=-a^2q^2, Q=q, T=-1$ into $\overline{\mathscr{P}(A, Q, T)}(F_s^{-2}F_tF_s^{-1}F_t)$ and multiplying by the unit $q^{2}a^{-3}$, one readily verifies that this coincides with the HOMFLY–PT polynomial of $L5a1$ computed in \cite{LM}
\[ \frac{1}{vz} - \frac{v}{z} - \frac{z}{v^{3}} + \frac{2z}{v} - vz + \frac{z^{3}}{v} \]
after setting $z=q-q^{-1}$ and $v=a$.

\appendix

\begin{appendices}

\end{appendices}

\printbibliography

\end{document}